\documentclass[12pt]{amsart}
\usepackage{graphicx}
\usepackage{amsfonts}
\usepackage{amssymb}
\usepackage{latexsym}
\usepackage{amscd}

\newtheorem{thm}{Theorem}[section]

\newtheorem{lem}[thm]{Lemma}
\newtheorem{prop}[thm]{Proposition}

\newenvironment{pf*}[1]{\proof[#1]}{\endproof}
\usepackage{euscript}

\usepackage[OT2,OT1]{fontenc}

\newcommand{\h}{\hat}

\newcommand{\cal}[1]{{\mathcal #1}}

\theoremstyle{definition}

\theoremstyle{remark}

\newcommand{\dist}{\operatorname{dist}}

\newcommand{\tl}{\tilde}
\newcommand{\wtl}{\widetilde}
\newcommand{\eps}{\epsilon}

\newcommand{\aaa}[1]{{{\mathbf{#1}}}}

\renewcommand{\Im}{\operatorname{Im}}

\numberwithin{equation}{section}

\newcommand{\cO}{{\cal O}}

\newcommand{\cI}{{\cal I}}
\newcommand{\cU}{{\mathcal U}}

\newcommand{\cB}{{\aaa B}}

\newcommand{\bT}{{\mathbf T}}

\newcommand{\cC}{{\cal C}}
\newcommand{\cAC}{{\cal A\cal C}}

\newcommand{\cR}{{\cal R}}

\newcommand{\CC}{{\Bbb C}}
\renewcommand{\AA}{{\Bbb A}}
\newcommand{\RR}{{\Bbb R}}
\newcommand{\TT}{{\Bbb T}}
\newcommand{\ZZ}{{\Bbb Z}}
\newcommand{\NN}{{\Bbb N}}
\newcommand{\DD}{{\Bbb D}}

\newcommand{\QQ}{{\Bbb Q}}

\begin{document}
\addtolength{\evensidemargin}{-0.7in}
\addtolength{\oddsidemargin}{-0.7in}

\title[KAM-renormalization and Herman rings for 2D maps]{KAM-renormalization and Herman rings for 2D maps }
\author{Michael Yampolsky}
\date{\today}

\begin{abstract}
In this note, we extend the renormalization horseshoe we have recently constructed with N.~Goncharuk for analytic diffeomorphisms of the circle to their small two-dimensional perturbations. As one consequence, Herman rings with rotation numbers of bounded type survive on a codimension one set of parameters under small two-dimensional perturbations.
\end{abstract}

\maketitle

\section{Foreword}
In a recent paper \cite{GoYa},   N.~Goncharuk and the author  constructed a renormalization operator for complex-analytic maps of the annulus, which leaves rigid rotations invariant. This operator is complex analytic and real-symmetric with respect to the natural Banach space structure given by the sup-norm, and its differential is compact. Furthermore, the set of Brjuno rotations is a hyperbolic invariant set of this operator, with one-dimensional unstable bundle. This set has a stable foliation of codimension one, whose leaves consist of maps which are analytically conjugate to Brjuno rotations.

The construction of the renormalization operator in \cite{GoYa} parallels the cylinder renormalization operator introduced by the author in \cite{Ya3}. The hyperbolic horseshoe for cylinder renormalization consists of analytic critical circle maps. In another recent work \cite{GaYa}, D.~Gaidashev and the author extended a subset of this horseshoe, corresponding to the rotation numbers of bounded type, to two-dimensional complex-analytic maps. In the real slice of the corresponding Banach manifold, the leaves of the stable foliation have codimension one, and consist of real-analytic dissipative maps of the annulus with a critical circle attractor, the dynamics on which is homeomorphically but not smoothly conjugate to the rigid rotation.

Finally, in \cite{GaRaYa}, a similar set of ideas was used by D.~Gaidashev, R.~Radu, and the author to prove that Siegel disk boundaries of highly dissipative H{\'e}non maps with golden-mean semi-Siegel fixed points are topological circles; the corresponding renormalization picture was constructed in \cite{GaYa1}.

This note synthesizes the above ideas, and adds some new ones, to extend the renormalization picture of \cite{GoYa} restricted to rotation numbers of bounded type to two-dimensional complex-analytic  dissipative maps (Theorem~\ref{th:hyperb3}). The maps populating the stable leaves possess analytic invariant circles, the dynamics on which is analytically conjugate to the rotation (Theorem~\ref{th:rotation}). In particular, this implies that Herman rings with rotation numbers of bounded type survive on a codimension one subset of parameters under small two-dimensional perturbations.

This note is not supposed to be a self-contained introduction to the subject. Rather, we assume that the reader is either familiar with the above quoted works, or will look up the relevant statements there -- what we do here is connecting the dots.

\section{Function spaces and  1D renormalization}
\subsection{Commuting pairs}
Denote $U_r(z_0)=\{|z-z_0|<r\}$; for a set $S\subset \CC$ set $$U_r(S)=\underset{z\in S}{\cup} U_r(z).$$

Let $\cB$ denote Brjuno rotation numbers. For $M\in\NN$ let $\bT_M$ denote the irrational rotation numbers of type bounded by $M$ (that is, $M$ is an upper bound for the terms in the continued fraction expansion of the rotation number).

Let us set
$$T_\theta(z)=z+\theta,\text{ and }\alpha(z)=T_1(z)=z+1.$$
Let us fix a real-symmetric topological disk $W\Supset [0,1]$. We define $\cC_W$ to be the space of maps $\beta$ such that:
\begin{itemize}
\item $\beta$ is a bounded analytic map in $W$, continuous up to the boundary;
\item $\alpha\circ \beta=\beta\circ \alpha$ where defined.
\end{itemize}
Equipped with the uniform norm,  $\cC_W$ is a Banach manifold (since $\beta(z)-z$ is a $1$-periodic function).

We will refer to pairs $(\alpha,\beta)$ where $\beta\in\cC_W$ as {\it normalized commuting pairs}.

The connection with the usual definition of commuting pairs (which we will refer to as {\it non-normalized} commuting pairs, to avoid confusion) \cite{GaYa} is as follows. Let $\zeta=(\eta,\xi)$ be a commuting pair of analytic diffeomorphisms. The model case to consider is a map $F$ of an annulus around $\TT$ which is close to an irrational rotation $T_\theta$, and
$$\eta=F^{q_n}\text{, and }\xi=F^{q_{n+1}}.$$
Now, let $\Psi$ be a locally conformal  solution (cf. \cite{GoYa}) of  the functional equation
$$\Psi^{-1}\circ \eta\circ \Psi=\alpha$$
and set $\beta=\Psi^{-1}\circ\xi\circ \Psi.$ Such a solution exists and can furthermore be chosen to analytically depend on $\eta$ (see  \cite{GoYa} for the construction, and also \cite{GaiYam} for a constructive version), associating a normalized commuting pair to a non-normalized one.

Renormalization $\cR$ on normalized commuting pairs is simply the cylinder renormalization defined in \cite{GoYa}, 
since $B$ naturally projects to a map of the quotient $W/_{z\sim z+1}$ and vice versa.
We thus have the following Renormalization Hyperbolicity Theorem \cite{GoYa}:
\begin{thm}
  \label{th:hyperb1}
  There exists a domain $W$ such that the following holds. The operator $\cR$ maps an open subset $\cU$ of $\cC_W$ to a pre-compact subset of $\cC_W$. The invariant set $\{T_\theta\;|\;\theta\in\cB\}\subset \cC_W$ is hyperbolic, with one-dimensional unstable direction. It has a codimension one strong stable foliation by analytic submanifolds, and every $\beta\in\cC_W\cap W^s(T_\theta)$ with $\theta\in\cB$
  is analytically conjugate to $T_\theta$ in a neighborhood of $[0,1]$. Furthermore, for every $M\in\NN$, hyperbolicity is uniform on the set $\{T_\theta\;|\;\theta\in\bT_M\}$.
%
%  The $\cR$-invariant
%  Let $\theta\in(0,1]$ be periodic under the Gauss map with period $k$. There exists a domain $W$ such that 
%    Moreover, $\cR$ has a periodic point of period $k$ of the form $F=(\text{Id},T_\theta)$. The linearization $D\cR^k|_F$ is a compact linear operator which has a single unstable eigenvalue (corresponding to changing $\theta$) and the rest of its eigenvalues lie inside $\DD$.
\end{thm}

\subsection{Almost commuting pairs}
We say that $(\alpha,\beta)$ is {\it a  normalized almost commuting pair} if
\begin{itemize}
\item $\beta$ is a bounded analytic map in $W$, continuous up to the boundary;
\item $[\alpha,\beta](z)\equiv \alpha\circ \beta(z)-\beta\circ \alpha(z)=o(z^2)$.
\end{itemize}
We denote the space of such maps $\beta$ by $\cAC_W$; as seen in \cite{GaYa} it is a Banach submanifold of the space of bounded analytic functions in $W$, continuous up to the boundary.
The definition of $\cR$ is naturally extended to these pairs; the image of an almost commuting pair is again an almost commuting pair.

Just as before, these pairs correspond to  conformal rescalings of {\it non-normalized} almost commuting pairs \cite{GaYa,GaRaYa} $\zeta=(\eta,\xi)$ with 
$$[\zeta](z)=\eta\circ \xi(z)-\xi\circ\eta(z)=o(z^2).$$

It turns out that almost commutativity improves under renormalization:
\begin{thm}
  \label{th:commutator1}
  Fix  $\tau\in(0,1)$, and let $M\in\NN$. There exist $\delta>0$, and $\ell\in\NN$ such that for every $\theta \in\cB$, there exists $\eps=\eps(\theta)$
    such that the following is true. 
  Let $\beta\in\cAC_W$ be $\eps$-close to $T_\theta$, and set $\nu=(\alpha,\beta)$. Then 
  \begin{equation}
    \label{commutator-estimate1}
    ||[\cR^\ell \nu]||^\infty_{U_\delta(0)}<\tau||[\nu]||^\infty_{U_\delta(0)}.
    \end{equation}
Moreover, suppose $\theta\in\bT_M$. Then $\eps$ can be chosen depending only on $M$. 

\end{thm}
It is easier to prove the estimate for non-rescaled pairs $\zeta=(\eta,\xi)$. 
We will need the following lemma:
\begin{lem}
\label{lem-comm}
Let $\zeta=(\eta,\xi)$, and denote $p\cR^\ell\zeta=(\eta_\ell,\xi_\ell)$ the $\ell$-th pre-renormalization of $\zeta$. Then,
\begin{equation}
  \label{eq-comm-comp}
  \eta_\ell\circ\xi_\ell=f_\ell\circ\eta\circ\xi\text{ and }\xi_\ell\circ\eta_\ell=f_\ell\circ\xi\circ\eta,
  \end{equation}
where $f_\ell$ is  the same composition of iterates of $\eta$ and $\xi$. In particular, the commutator $[p\cR^\ell\zeta]$ has the form
\begin{equation}
  \label{eq-comm}
  [p\cR^\ell\zeta]=f_\ell\circ\eta\circ\xi-f_\ell\circ\xi\circ\eta.
\end{equation}

\end{lem}
\begin{proof}
  The proof of (\ref{eq-comm-comp}) is easily supplied by induction. Indeed, if
  we assume that it holds for some $\ell$, then
  %set $p\cR^\ell\zeta=(\eta_\ell,\xi_\ell)$ and assume that
  %$$[p\cR^\ell\zeta]=f_\ell\circ \eta_\ell\circ \xi_\ell-f_\ell\circ \xi_\ell\circ \eta_\ell,$$
  %then
  $$\xi_{\ell+1}\circ\eta_{\ell+1}=\xi^r_\ell\circ\eta_\ell\circ\xi_\ell=\xi^r_\ell\circ f_\ell\circ \eta\circ\xi,\text{ and }
  \eta_{\ell+1}\circ\xi_{\ell+1}=\xi_\ell\circ\xi^r_\ell\circ\eta_\ell=\xi^r_\ell\circ f_\ell\circ \xi\circ\eta,$$
  so we get the desired claim by setting
  $$f_{\ell+1}=-\xi_\ell^r\circ f_\ell.$$
  \end{proof}

%Now, let the pre-renormalization
%\cite{GaYa,GaRaYa}
%$p\cR^\ell\zeta=(\eta_\ell,\xi_\ell)$ and
Set $\lambda_\ell=|\xi_\ell(0)|$.
  In view of Lemma~\ref{lem-comm}, we have the following first-order estimate:
  $$[p\cR^\ell\zeta]=\eta_\ell\circ\xi_\ell-\xi_\ell\circ\eta_\ell\sim f_\ell'(\eta\circ\xi(0))[\zeta]\sim f_\ell'(\eta\circ\xi(0))\cdot cz^2.$$
  Assuming the pair is sufficiently close to linear, the derivative $f_\ell'$ is uniformly bounded. 
  Thus,
  $$|[\cR^\ell\zeta](z)|\sim \lambda_\ell |f_\ell'(\eta\circ\xi(0))|\cdot |c||z|^2\lesssim\tau |z|^2.$$
  for $\ell$ large enough and $\zeta$ sufficiently close to linear.
The statement for normalized almost commuting pairs clearly follows.

  As a corollary of Theorem~\ref{th:commutator1}, the operator $\cR$ in $\cAC_W$ does not have any new unstable directions:
  
\begin{thm}
  \label{th:hyperb2}
There exists a domain $W$ such that the following holds. The operator $\cR$ maps an open subset $\wtl\cU$ of $\cAC_W$ to a compact subset of $\cAC_W$. Its differential at every point is a compact linear operator. Fix $M\in\NN$. The invariant set $\{T_\theta\;|\;\theta\in\bT_M\}\subset \cAC_W$  is uniformly hyperbolic, with one-dimensional unstable direction. It has a codimension one strong stable foliation by analytic submanifolds.
  \end{thm}
\begin{proof}
  The map $\beta\in\cAC_W$ projects from the interval $[0,1]$ to a smooth map of the unit circle which is analytic at every point except one. 
When the commutator $[\alpha,\beta]$ is small, this map can be extended to a quasiregular map of a neighborhood with a small dilatation. Applying Measurable Riemann Mapping Theorem produces a conformal map of an annulus around the circle, whose lift $\tl \beta$ is close to $\beta$: the norm of the distance is comparable to the norm of $[\alpha,\beta]$. Since $\tl \beta$ commutes with $\alpha$, Theorem~\ref{th:commutator1} implies that the operator $\cR$ contracts the distance to $\cC_W$.
The statement follows by Theorem~\ref{th:hyperb1}.
\end{proof}

\section{Definition of 2D renormalization}
\subsection{Preliminaries: multi-indices of renormalization}
We follow the notation of \cite{GaRaYa}. Namely, let
us consider the space $\cI$ of multi-indices $\bar s=(a_1,b_1,a_2,b_2,\ldots,a_m,b_m)$ where $a_j\in \NN$ for $2\leq m$, $a_1\in\NN\cup\{0\}$,
$b_j\in\NN$ for $1\leq j\leq m-1$, and $b_m\in\NN\cup\{0\}$. 

%We introduce a partial ordering on multi-indices:
%$\bar s\succ \bar t$ if $\bar s=(a_1,b_1,a_2,b_2,\ldots,a_m,b_m)$, $\bar t=(a_1,b_1,\ldots,a_k,b_k,c,d)$, where $k<m$ and 
%either $c< a_{k+1}$ and $ d=0$ or $c=a_{k+1}$ and $d< b_{k+1}$.

For a pair of maps $\zeta=(\eta,\xi)$ and $\bar s $ as above we will denote 
$$\zeta^{\bar s}\equiv\xi^{b_m}\circ\eta^{a_m}\circ\cdots\circ\xi^{b_2}\circ\eta^{a_2}\circ\xi^{b_1}\circ \eta^{a_1}.$$
Similarly, 
$$\zeta^{-\bar s}\equiv (\zeta^{\bar s})^{-1}=(\eta^{a_1})^{-1}\circ(\xi^{b_1})^{-1}\circ\cdots\circ(\eta^{a_m})^{-1}\circ (\xi^{b_m})^{-1}.$$

%Consider a pair $\zeta\in W^s(\zeta_\lambda)\subset \cB(Z,W)$. 

Consider the $n$-th pre-renormalization of $\zeta$:
$$p\cR^n\zeta=\zeta_{n}=(\eta_{n}|_{Z_n},\xi_{n}|_{W_n}),$$
where $Z_{n} = {\mathbf l}_{n}(Z)$, $W_{n} = {\mathbf l}_{n}(W)$, and 
%Set 
\begin{equation}\label{alpha_n}
{\mathbf l}_n(z)=\eta_n(0) z.
\end{equation}

We define  $\bar s_n,\bar t_n\in\cI$ to be such that
$$\eta_{n}=\zeta^{\bar s_n},\text{ and }\xi_{n}=\zeta^{\bar t_n}.$$
A straightforward induction shows:
\begin{lem}
\label{multi1}
For $n\geq 1$, let $\bar r=\bar s_n$ or $\bar t_n$. Write $\bar r=(a_1,b_1,a_2,b_2,\ldots,a_{m_n},b_{m_n})$. Then $b_{m_n}=0$, and either
$$a_{m_n}\geq 2$$
or
$$a_{m_n}=b_{m_n-1}=1.$$
Furthermore, if $\bar s_n$ ends in $\ldots,1,1,0$ then so does $\bar t_n$.
\end{lem}
Let $\bar s_n$ be given by $(a_1,b_1,a_2,b_2,\ldots,a_{m_n},0)$.
We denote 
\begin{eqnarray}
\nonumber \hat s_n&=&\left\{ (a_1,b_1,a_2,b_2,\ldots,a_{m_n}-2,0), \ a_{m_n} \ge 2  \atop (a_1,b_1,a_2,b_2,\ldots,0,0,0), \ a_{m_n}=1 \right. ,\\
\nonumber \phi_0(x)&=&\left\{ \eta^2, \ a_{m_n} \ge 2  \atop \eta \circ \xi, \ a_{m_n}=1 \right. .
\end{eqnarray}
Define $\hat t_{n}$ in an identical way to $\hat s_{n}$. Then $p \cR^n\zeta$ can be written as
$$ p \cR^n\zeta=\phi_0 \circ  \left(\zeta^{\hat s_n}, \zeta^{\hat t_n } \right).$$
\subsection{Renormalization of two-dimensional maps}
Let us fix $W$ as in Theorem~\ref{th:hyperb2}, and let $\Omega=W\times W$. We also let $\Gamma=U_R([0,1])\times U_R([0,1])$ for $R>0$.

We denote $\cO_{\Gamma,\Omega}$ the Banach space of pairs of bounded analytic functions continuous up to the boundary
$$F=(F_1:\Gamma\to\CC^2,F_2:\Omega\to\CC^2)$$ equipped with the norm
\begin{equation}
\label{eq:unormm}\| F\|= \frac{1}{2} \left(   \sup_{(x,y) \in \Gamma}|F_1(x,y)|+  \sup_{(x,y)\in \Omega} |F_2(x,y) | \right).
\end{equation}
%% We let $\cO(\Omega,\Gamma,\delta)$ stand for the $\delta$-ball around the origin in this Banach space. 

We define a transformation $\iota$ which sends the pair $\nu=(\alpha,\beta)$ to
the  pair of functions $\iota(\nu)$:
\begin{equation}
\label{eq:embed}
\left(\left(x \atop  y\right)\mapsto \left(\alpha(x) \atop \alpha(x) \right), 
\left(x \atop  y\right)\mapsto \left(\beta(x) \atop \beta(x) \right)   \right).
\end{equation}

Next, we follow \cite{GaYa,GaRaYa} to define renormalization of small (of size $O(\delta)$ for some small $\delta>0$) two-dimensional perturbations of one-dimensional almost commuting pairs.
Namely, we consider pairs of maps  of the form
\begin{eqnarray}
\label{eq:A} A(x,y)&=&(a(x,y),h(x,y))=(a_y(x),h_y(x)),\\
\label{eq:B} B(x,y)&=&(b(x,y),g(x,y))=(b_y(x),g_y(x)), 
\end{eqnarray}
such that
\begin{itemize}
\item[1)] the pair $(A(x,y),B(x,y))$ is in a $\delta$-neighborhood  of  $\iota(\cU)$,
\item[2)]  $(h,g)$ are such that  $|\partial_x h(x,0)|>0$ and  $|\partial_x g(x,0)|>0$,  and 
$$(h(x,y)-h(x,0),g(x,y)-g(x,0))\text{ have norms  bounded by }\delta.$$ 
\end{itemize}
Suppose that $\Sigma=(A,B)$ as above is a a small perturbation of a pair $\iota(\nu)$ where $\nu=(\alpha,\beta)$.

Set
$$\h p \cR^n \Sigma = \left(F \circ \Sigma^{\h s_n} \circ A, F \circ \Sigma^{\h t_n} \circ A \right),$$
where $F=A$ if $a_n \ge 2$ and $F=B$ if $a_n=1$.

We will denote $$\pi_1(x,y)=x\text{ and }\pi_2(x,y)=y.$$

Set
$$\phi_y(x)=\phi(x,y) :=\left\{ \pi_1 A^2(x,y), \ a_n \ge 2  \atop \pi_1 A \circ B(x,y), \ a_n=1 \right.$$
Furthermore, set
$$q_z(x) \equiv q(x,z)=\pi_2 F(x,z)= \left\{ h_z(x), \ a_n \ge 2  \atop g_z(x), \ a_n=1 \right.$$
Also, set
\begin{equation}
\nonumber w_z(x) \equiv w(x,z) := q_{z}\left(\phi_{z}^{-1}( x) \right).
\end{equation} 
Notice, that $\partial_z  w_z(x)$ and $\partial_z w^{-1}_z(x)$ are functions whose uniform norms are $O(\delta)$.

\noindent
Define the following transformation:
\begin{equation}\label{eq:Htransform}
H_{\Sigma}(x,y) = (a_y(x),w_{q_0^{-1}(y)}^{-1}(y)),
\end{equation}
We have
\begin{equation}
\label{eq:AHinv} A \circ H_{\Sigma}^{-1}(x,y)=(x,h(\alpha^{-1}(x),y))+O(\delta).
\end{equation}

We use $H_{\Sigma}(x,y)$ to pull back $\hat p \cR^n \Sigma$ to a neighborhood of definition of the $n$-th pre-renormalization of a pair $(\alpha,\beta)$:
$$p \cR^n \Sigma=(\bar A, \bar B) = H_{\Sigma}\circ F \circ \left( \Sigma^{\h s_n},  \Sigma^{\h t_n} \right) \circ A \circ H_{\Sigma}^{-1}(x,y).$$
The proof of the following statement is straightforward (compare with \cite{GaYa}):
\begin{prop}
  \label{prop:acproj}For evey $n$ there exists $\delta$ small enough such that the following holds. There exists a unique triple $(d_0,d_1,d_2)\in\CC^3$ which analytically depends on the pair $\Sigma$ such that
the almost commutation condition
  \begin{equation}
    \label{aceq2D}
    \pi_1(\wtl A\circ\wtl B)(x,0)-\wtl B\circ \wtl A(x,0))=o(x^2)
  \end{equation}
  holds for
  the pair
  \begin{equation}
    \label{acproj}
    (\wtl A,\wtl B)\equiv (\bar A,\bar B)+\left(
    \left(
    \begin{array}{c}
      0\\
      0
      \end{array}
    \right),
\left(
    \begin{array}{c}
      d_0+d_1x+d_2x^2\\
      d_0+d_1x+d_2x^2
      \end{array}
    \right)
    \right)
    \end{equation}
  \end{prop}
Note that:
\begin{prop}
If $A$ and $B$ commute then $$(\wtl A,\wtl B)=(\bar A,\bar B).$$
\end{prop}
Set
$$\wtl \alpha\equiv \pi_1\wtl A(x,0),$$
and let $\psi(x)$ be a solution of the functional equation $$\psi^{-1}\circ\wtl\alpha\circ\psi(x)=T(x)$$
chosen to analytically depend on $\wtl\alpha$, satisfying $\psi(0)=0$. Set
$$ \Psi(x,y)=(\psi(x),y)$$
and
$$\hat A\equiv  \Psi^{-1}\wtl A(\Psi(x,y)),\; \hat B\equiv \Psi^{-1}\wtl B(\Psi(x,y))$$
In the same way as in \cite{GaYa}, we have 
\begin{thm}
  \label{th:sizeperturb}For every $n,M\in\NN$, 
  there exist constants  $\delta_0>0$, $C>0$, a topological disk $\hat W\supset W$,  and a neighborhood $\hat \cU\in\cO_{\Gamma,\Omega}$ of $\iota\{(T,T_\theta)\;|\;\theta\in\cB\}$ such that the following holds. Set $\hat\Omega=\hat W\times \hat W$.  Assume $\delta<\delta_0$, and $\Sigma=(A,B)\in \hat \cU$. Then
  $(\hat A,\hat B)$ lies in $\cO_{\Gamma,\hat\Omega}$, and
  $$\dist((\hat A,\hat B), \iota(\cAC_W))<C\delta^2.$$
  
\end{thm}
Let us fix the quantifiers as in Theorem~\ref{th:sizeperturb} and define the {\it renormalization operator of level} $n$ as
$$\cR_n:(A,B)\mapsto (\hat A,\hat B).$$
Theorems \ref{th:hyperb2} and \ref{th:sizeperturb} imply the following:
\begin{thm}
  \label{th:hyperb3}
  The operator $\cR_n$ is a compact analytic operator of an open neighborhood $\hat \cU\in\cO_{\Gamma,\Omega}$ of $\iota\{(T,T_\theta)\;|\;\theta\in\bT_M\}$ to
  $\cO_{\Gamma,\Omega}$. Its differential is a compact linear operator. The invariant set $\iota\{(T,T_\theta)\;|\;\theta\in\bT_M\}$ is uniformly hyperbolic, with one dimensional unstable direction. It has a codimension one strong stable foliation by analytic submanifolds.
 \end{thm}

\section{Rotation domains}
\label{sec:rotation}
For $r>0$ set $$\AA_r\equiv \{|\Im z|<r\}/\ZZ.$$
We say that a dissipative H\'enon-like map $H$ has a {\it rotation domain} if there exists an $H$-invariant domain $C$ and an analytic isomorphism
$$\Phi:C\to \AA_r\times \DD\text{ for some }r>0,$$
which conjugates $H$ to the linear map
$$L(x,y)=(e^{2\pi i\theta}x,sy)\text{ with }\theta\in\RR\setminus \QQ.$$

Furthermore, let us say that a pair $(A,B)\in\cO_\Omega$ is a renormalization of a map $H$ if there exists $m\in\NN$ and a linear map $\Lambda$ such that
$$(A,B)=\Lambda\circ(H^{q_m},H^{q_{m+1}})\circ \Lambda^{-1}.$$
Finally, for a Brjuno number $\theta\in(0,1)$ let us denote
$$Y_m(\theta)=\sum_{j=0}^m\theta_{-1}\theta_0\theta_1\cdots\theta_{j-1}\log\frac{1}{\theta_j},\text{ where }\theta_{-1}=1,\;\theta_0=\theta\text{, and }\theta_{j+1}=\{1/\theta_j\}.$$
Note that $Y_\infty(\theta)<\infty$ is the Brjuno-Yoccoz function of $\theta\in\cB$.

\begin{thm}
  \label{th:rotation}
  Suppose a renormalization of $H$ lies in the strong stable foliation of
  $$\iota\{(T,T_\theta)\;|\;\theta\in\bT_M\}$$
constructed in Theorem~\ref{th:hyperb3}.
  Then $H$ possesses a rotation domain.
\end{thm}

\begin{proof}
  Denote $(A,B)$ the renormalization of $H$ in the strong stable manifold of $\iota(T,T_\theta)$ with $\theta\in\cB$, and
  let $(A_k,B_k)=(p\cR^n)^k(A,B)$. Consider a shadowing $p\cR_n$-orbit of a normalized almost-commuting pair $(T,g)$ (note that this is a pre-renormalization orbit, which thus consists of non-normalized pairs),
  $$p\cR^n:(T,g_k)\mapsto (T,g_{k+1}),\; g_0=g$$
  so that
  $$\dist(G_k,B_k)\lesssim\delta^{2^k}\text{ where }G_k=\iota(g_k).$$
Let $h_k$ be the corresponding linearizing map
  $$h_k\circ g_k\circ h_k^{-1}(z)=1.$$
  Let $H_k$ be the lift of $(h_k(x),h_k(y))$ to the domain of definition of $(A,B)$ via the renormalization microscope (see \cite{GaRaYa} for the details of the construction).
%By Yoccoz's result \cite{Yoccoz2002}, \S 4.2, r
Replacing $(A,B)$ with $\cR_n^l(A,B)$ for a sufficiently large $l$ if needed, we can guarantee that
$H_k$ is defined in a polydisk, whose $x$-projection is a rectangle of height
$$|\Im z|>\Delta -C-Y_{kn}(\theta),$$
where $C$ is a constant independent of $k$ and $\Delta$ is large compared to $C+Y_\infty(\theta)$.
Passing to a limit $H_\infty$ completes the proof.
  
  \end{proof}

\bibliographystyle{plain}
\bibliography{biblio}
\end{document}